\documentclass[12pt, reqno]{amsart}
\usepackage{amsfonts}
\usepackage{amsmath,amsthm,mathrsfs}
\usepackage{amssymb}
\usepackage{enumitem}
\usepackage{tikz}
\usepackage{bm}
\usepackage{url}
\usepackage[normalem]{ulem}

\usepackage[colorlinks=true,linkcolor=blue,citecolor=blue,urlcolor=blue,pdfborder={0 0 0}]{hyperref}

\usepackage{soul}

\usepackage{graphicx}
\usepackage{subcaption}

\usepackage{xcolor}
\usepackage[capitalize]{cleveref}

\usepackage{biblatex}
\theoremstyle{plain}

\newtheorem{thm}{Theorem}[section]

\newtheorem{lemma}[thm]{Lemma}

\newtheorem{claim}[thm]{Claim}

\theoremstyle{definition}
\newtheorem{remark}[thm]{Remark}

\usepackage[legalpaper, margin=1.2in]{geometry}

\usepackage{fancyvrb}
\usepackage{mathtools}

\newcommand{\R}{\mathbb{R}}

\renewcommand{\P}{\mathbb{P}}

\DeclareMathOperator{\E}{\mathbb{E}}

\newcommand{\RR}{\mathbb{R}}

\newcommand{\Z}{\mathbb{Z}}

\newcommand{\pot}{\Psi}

\newcommand{\eq}{equation}
\newcommand{\be}{\begin{\eq}}
\newcommand{\ee}{\end{\eq}}

\renewcommand{\and}{\mathrm{\ and\ }}

\renewcommand{\o}{o}

\newcommand{\G}{{\sf G}}
\newcommand{\g}{{\sf g}}

\begin{document}

\title{Every recurrent network has a potential tending to infinity}
\author{Asaf Nachmias}
\address{Department of Mathematical Sciences, Tel Aviv University, Israel.}
\email{asafnach@tauex.tau.ac.il}
\author{Yuval Peres}
\address{Beijing Institute of Mathematical Sciences and Applications, Beijing, China.}
\email{yperes@gmail.com}

\begin{abstract} A rooted network consists of a connected, locally finite 
  graph $G$, equipped with edge conductances and a distinguished vertex $\o$.   A nonnegative function on the vertices of $G$ which vanishes at $\o$, has Laplacian $1$ at $\o$, and is harmonic at all other vertices is called a potential. We prove that every infinite recurrent rooted network admits a potential tending to infinity. This is an analogue of classical theorems due to Evans and Nakai in the settings of Euclidean domains and Riemannian surfaces.

\end{abstract}
\begingroup
\def\uppercasenonmath#1{} 
\let\MakeUppercase\relax 
\maketitle
\endgroup


\section{Introduction}\label{sec:intro}

A {\em rooted network} $(G,c,\o) $ consists of a     connected, locally finite graph $G=(V,E)$, endowed with positive edge weights  $c:E\to (0,\infty)$  and a distinguished vertex  $\o \in V$. 
The corresponding {\em network random walk} with transition probabilities $p(x,y)=c_{xy}/c_x$ is reversible, with reversing measure $c_x=\sum_{z\sim x}c_{xz}$. Define the network {\bf Laplacian} of a function $f:V\to\RR$ by 
$$ \Delta f (v) = \sum_{x \sim v} c_{vx} \big(f(x) - f(v)\big) \, .$$

 A {\bf potential} on  $(G,c,\o )$ is a function $h:V\to [0,\infty)$ that is harmonic off $\o$ (i.e.,  $\Delta h(v) = 0$ for every vertex $v\neq \o$), 
normalized so that $h(\o)=0$ and $\Delta h(\o) =1$.
In every infinite rooted network there is   at least one potential, see \cref{claim:compact}. Our  focus is on recurrent networks.



We say that a function $h:V\to [0,\infty)$ {\bf tends to infinity} and write  $h\to \infty$, if for all $M>0$ the set $\{ v \in V : h(v) \leq M \}$ is finite. Observe that on $\Z$, the potentials $h_1(n)=\max(n,0)$ for $n\in \Z$ and $h_2(n)=\max(-n, 0)$ do \emph{not} tend to infinity, but any non-trivial convex combination of them does. On $\Z^2$, on the other hand, it is  well known  that there is a unique potential $h$ with $h(v) = \Theta(\log |v|)$ as $v \to \infty$,   see \cite[Sections 12 and 31]{Spitzer}. Recently there has been renewed interest in recurrent potential theory, motivated by close connections to unimodular random networks and uniform spanning trees in work of  Berestycki, van Engelenburg and Hutchcroft, see  \cite{BereEngel,EngelHutch}.

\begin{thm}\label{conj:main} On any infinite recurrent rooted network there exists a potential  which tends to infinity.
\end{thm}

The earliest theorem of this type is due to G.C.~Evans \cite{Evans36} who showed that every compact polar set $K$ in $\R^d$ supports a probability measure such that its Newtonian potential is infinite on $K$. Using Evans' approach, Nakai \cite{Nakai62} proved that every parabolic noncompact Riemannian surface admits a function which is harmonic off a compact set and tends to infinity (i.e., its level sets are compact). In Nakai's proof, the set $K$ is replaced by the Stone-\v Cech boundary of the surface. Nakai's argument can be adapted to prove \cref{conj:main}, but the use of the Stone-\v Cech compactification makes this argument non-constructive.  Here we present a more constructive probabilistic approach, using the minimax theorem. 

\begin{remark} \cref{conj:main} implies that if a rooted network admits a unique potential, then it must tend to infinity. This answers a question raised following \cite[Prop. 5.1]{BereEngel}.
\end{remark}

\begin{remark} It is well known that the existence of a potential tending to infinity implies recurrence. Indeed, let $h$ be a potential, let $\{X_t\}$ be the network random walk and consider the hitting time $\tau_\o=\min\{ n \geq 0 : X_n = \o\}$. Then $h(X_{n\wedge \tau_\o})$ is a nonnegative martingale, whence it almost surely converges; this     contradicts transience if $h\to \infty$.
\end{remark}

\section{Preliminaries}
\subsection{Dipoles and potentials} The {\bf Green kernel} $\G_\o(x,y)$ of the random walk $\{X_n\}_{n \geq 0}$ killed at $\o$ is defined by
$$ \G_\o(x,y) := \sum_{n=0}^\infty \P_x( X_n = y, \,\, \tau_\o > n) \, .$$
The corresponding {\bf Green density},  $${\sf g}_\o (x,y) = \frac{\G_\o(x,y)}{c_y} \, ,$$
is symmetric by reversibility. A {\bf dipole} from $\o$ to another vertex $y$ in a network $(G,c)$ is a function $f:V\to [0,\infty)$, for which $f(\o)=0$ and 
$$\forall v \in V \qquad \Delta f (v) = {\bf 1}_\o(v)  - {\bf 1}_y(v) \, .$$

\begin{claim}\label{claim:GreenIsDipole} In a recurrent network, the Green density $g_\o(\cdot,y)$ for the killed random walk is a dipole from $\o$ to $y$. Moreover,  $\sup_y g_\o(x,y)<\infty$ for each $x \in V$.   
\end{claim}
\begin{proof}
First observe that $g_\o(\cdot,y)$ is harmonic off $\{\o,y\}$ and satisfies $\g_\o(\o, y)=0$. Second, $\Delta \g_\o(\cdot,y) (y) =-1$, since
\begin{equation} \label{eq:laplace-y}
     \sum_{x \sim y} c_{xy} (\g_\o(y,y)-\g_\o(x,y)) = \G_\o(y,y) - \sum_{x \sim y} {c_{xy} \over c_y} \G_\o(x,y) = 1 \, .
    \end{equation}
Also,
\be \label{eq:dipole11}  \Delta \g_\o(\cdot,y)(\o) = \sum_{x\sim \o}   c_{\o x} g_\o(x,y)= \sum_{x\sim \o} \G_o(y,x) \frac{c_{\o x}}{c_x} =1 \,,
\ee
where the last equality is obtained by summing, over  $n \ge 0$, the identity
$$
\sum_{x\sim \o}  \P_y( X_{n} = x  , \,\, \tau_\o > n)\frac{c_{\o x}}{c_x} =  \P_y( X_{n+1} = \o, \,\, \tau_\o > n)  = \P_y(   \tau_\o =n+1)  \, 
$$
and using recurrence. Thus, $g_\o(\cdot,y)$ is a dipole. 


Lastly, the identity \eqref{eq:dipole11}  implies that $\g_\o(x,y) \le c_{ox}^{-1}$ for every $x \sim \o$. More generally, since $\g_\o(\cdot,y)$ is superharmonic off $\o$, for any $x \ne \o$ we have $\g_\o(x,y) \le \prod_{i=0}^\ell c_{x_{i-1}x_i}^{-1}$, where $x_0,\dots,x_\ell$ is a simple path from $\o$ to $x$. 
\end{proof}

\begin{claim}\label{claim:compact} The space of potentials on a recurrent infinite rooted $(G,c,\o)$ is a nonempty, compact, convex subset of $\R^V$ endowed with the product topology. 
\end{claim}
\begin{proof} 
By \cref{claim:GreenIsDipole}, for any sequence of vertices $y_n\to \infty$, Cantor diagonalization yields a subsequential pointwise limit of $\g_\o(\cdot, y_n )$, and any such limit is a potential.  Next, for any potential $h$ we have $h(x)\leq \prod_{i=0}^\ell c_{x_{i-1}x_i}^{-1}$, where $x_0,\dots,x_\ell$ is a simple path from $\o$ to $x$, by the argument in the proof of \cref{claim:GreenIsDipole}, so the set of potentials is a closed subset of a compact product space.
Finally, convexity is obvious. 
\end{proof}

\subsection{The Doob transform determined by a potential}
Any potential $h$ defines a Markov chain on the state space $S_h = \{v : h(v)>0\}$ via the Doob $h$-transform that has transition probabilities
$$ p^h(x,y) = {p(x,y) h(y) \over h(x)} \, .$$
The corresponding Markov chain $\{Y_n\}_{n \geq 0}$ is called the {\bf $h$-process}. Write $\P^h_\mu$ for the law of the $h$-process with initial distribution $\mu$ and $\E^h_\mu$ for the corresponding expectation.


For any finite sequence $\gamma=(\gamma_1,\ldots,\gamma_k)$ in $V$, we write 
$\P_{\gamma_1}(\gamma):=\prod_{i=2}^k p(\gamma_{i-1},\gamma_i) \,.$
If $\gamma$ is contained in $S_h$, we define 
$\P^h_{\gamma_1}(\gamma)$  similarly using $p^h(\gamma_{i-1},\gamma_i)$. We have the telescoping product
\be \label{eq:hprocessvsSRWSimple} \P^h_{\gamma_1}(\gamma) = \P_{\gamma_1}(\gamma) {h(\gamma_k) \over h(\gamma_1)} \, .\ee
We will use the initial distribution $\mu_h$ defined by
\be\label{def:mu}\mu_h(v) = c_{\o v} h(v) \quad \text{for all} \; \; v \sim \o  \, .\ee 
Observe that
\be \label{eq:hprocessvsSRW} \P^h_{\mu_h} (\gamma) = c_{\o \gamma_1} h(\gamma_1) \P^h_{\gamma_1}(\gamma) =  c_{\o \gamma_1} \P_{\gamma_1}(\gamma) h(\gamma_k) = c_\o \P_\o(\o\gamma) h(\gamma_k)  \, .\ee

There is a standard relation between the Green kernels of the $h$-process and the network random walk killed at $\o$.
\begin{claim} \label{claim:GhMu} Let $h$ be a potential on a recurrent rooted network $(G,c,\o)$. Then the $h$-process   $\{Y_n\}_{n \geq 0}$ is transient. Moreover,  its Green kernel $\G^h(x,y)=\sum_{k \geq 0} \P^h_x( Y_k = y) $, defined for $x,y\in S_h$, satisfies
\be\label{eq:XvsY} \G^h(x,y) = \G_\o(x,y) \frac{h(y)}{h(x)} \, ,\ee
and for all $v\in S_h$
\be\label{eq:Gh} \G^h({\mu_h},v) = h(v) c_v \, ,\ee
where $\G^h({\mu},v):=\sum_x \mu(x) \G^h(x,v)$.
\end{claim}
\begin{proof} Transience will follow    once we show $G^h(x,y)<\infty$. To prove \eqref{eq:XvsY}, we sum \eqref{eq:hprocessvsSRWSimple}  over all paths from $x$ to $y$ that avoid $\o$. Note that by harmonicity of $h$ off $\o$, any such path is contained in $S_h$. Next, \eqref{eq:Gh} follows from \eqref{eq:XvsY} by
$$ \G^h({\mu_h},y) =\sum_v c_{\o v} h(v) \G^h(v,y) = h(y) \sum_v c_{\o v} \G_\o(v,y) = h(y) c_y \, ,$$ using \eqref{eq:dipole11} and the symmetry of $\g_\o(\cdot,\cdot)$ in the last step.
\end{proof}

\begin{claim}\label{lem:hProcessInfty} Let $h$ be a potential for a  recurrent rooted network $(G,c,\o)$ and consider the $h$-process $\{Y_n\}$ started at $v \in S_h$. Then for every $M >0$, 
$$ \P^h_v \big( h(Y_\ell) >M \bigr) \to 1 \quad \text{as} \quad \ell \to \infty \, .$$
\end{claim}
This was also observed in \cite[Lemma 5]{EngelHutch}. 
\begin{proof}
   Recall that $\{X_n\}$ is the network random walk. By \eqref{eq:hprocessvsSRWSimple}, we have 
   $$\P_v^h\big( h(Y_\ell) \le M \bigr)=\E_v\Bigl[\frac{h(X_\ell)}{h(v)} \, {\bf 1}_{\{h(X_\ell) \le M \; \, \text{and} \; \,\tau_\o  \ge \ell\}}  \Bigr] \le \frac{M}{h(v)} \, \P_v( \tau_\o  \ge \ell) \,,$$  
   which tends to $0$ as $\ell \to \infty$ by  recurrence. 
\end{proof}


\section{Proof of main theorem} 
\label{sec:MainProof}

Let $d(\cdot,\cdot)$ denote  graph distance. For any integer $r>0$, we set
$$ B(\o,r) = \{ w: d(\o,w) < r\} \quad \textrm{and} \quad \partial B(\o,r) = \{ w : d(\o,w)=r\} \, ,$$
 Next  we state a formally weaker version of  \cref{conj:main},
where we can choose a different potential for each radius, and show that it implies the theorem. 

\begin{lemma}\label{lem:key} Let $\pot$ be the set of potentials on a recurrent infinite rooted network $(G,c,\o)$. Then
\be\label{eq:keylem} M(r):= \max_{\psi \in \pot} \min_{v \in \partial B(\o,r)} \psi(v) \longrightarrow \infty \qquad \textrm{{\rm as} } r \to \infty \, .\ee
\end{lemma}

\begin{remark} \label{remark:MR} 
For every integer $r>0$ and $\psi\in\pot$, the minimum of $\psi$ on $V \setminus B(\o,r)$ is attained on $\partial B(\o,r)$. 
Indeed, let $\{X_t\}$ be the network random walk started at $w \not \in B(\o,r)$ and let $\tau_r$ be the hitting time of $\partial B(\o,r)$. Then $\psi(X_{t\wedge \tau_r})$ is a positive martingale, so by recurrence and optional stopping,
$$ \psi(w) \geq \E_w [ \psi(X_{\tau_r})] \geq  \min_{v \in \partial B(\o,r)} \psi(v)\, .$$
Consequently, $\displaystyle M(r)=\max_{\psi \in \pot} \min_{w \in V\setminus B(\o,r)} \psi(w)$ is non-decreasing in $r$.   
\end{remark}


\medskip

\begin{proof}[{\bf Proof of \cref{conj:main} given \cref{lem:key}}] For each integer $r$, let $\psi_r$ be a potential where the maximum in the definition \eqref{eq:keylem} of $M(r)$ is attained. 
For every $n\geq 1$, choose $r_n$ such that $M(r_n)\geq n2^n$. By \cref{claim:compact},
$$ h := \sum_{n=1}^\infty 2^{-n} \psi_{r_n}$$
is a potential. Next, if $w\notin B(\o,r_n)$, then $h(w) \geq 2^{-n} \psi_{r_n}(w) \geq n$. Thus, for each $n$, the set $\{v \in V : h(v) < n\}$ is finite. 
\end{proof}

\subsection{Preliminary lemmas} The following lemma shows that if the dipoles $\g_\o(\cdot,v_n)$ converge to a potential $h$, then the network random walks conditioned on $\tau_{v_n} < \tau_\o$ converge weakly to the $h$-process. We actually need this for dipole mixtures. Given a probability measure $\eta$ on $V$, write $$ \g_\o(\cdot, \eta) := \sum_v \eta(v) \g_\o(\cdot, v) \, .$$ 

\begin{lemma} \label{lem:PathsConverge} Let $(G,c,\o)$ be a recurrent rooted network. Let $r_n\to \infty$ be sequence of integers and for each $n$ let $\eta_{r_n}$ be a probability measure supported on $\partial B(\o,r_n)$. Assume that the sequence $\g_\o(\cdot,\eta_{r_n})$ converges pointwise to a potential $h$ as $n\to \infty$. Then for any finite path $\gamma$ in $S_h$ starting at a neighbor of $\o$, we have
\be\label{eq:pathsConverge} \sum_{v} \eta_{r_n}(v) \P_\o( \o \gamma \mid \tau_{v} < \tau_\o^+) \longrightarrow \P_{\mu_h}^h(\gamma) \quad \text{as}    \quad n \to \infty \,.\ee

\end{lemma}
\begin{proof}
First note that for $v,w$ different from $\o$, we have
$\g_\o(w,v)=\P_w(\tau_v <\tau_\o)\g_\o(v,v)$ and  by path reversal,
$$ c_\o\P_\o(\tau_v <\tau_\o^+)= 
c_v \P_v(\tau_\o <\tau_v^+)= c_v \G_\o(v,v)^{-1} =\g_\o(v,v)^{-1}\,.$$
 Therefore, if $\gamma=(\gamma_1,\dots,\gamma_\ell)$   does not contain the vertex $v$, then  
 \be \label{dipoledoob}
\P_\o( \o \gamma \mid \tau_{v} < \tau_\o^+)=\P_\o( \o \gamma) \frac{ \P_{\gamma_\ell}(   \tau_{v} < \tau_\o) }{\P_{\o}(   \tau_{v} < \tau_\o^+)}=
c_\o \P_\o( \o \gamma)  \g_\o(\gamma_\ell,v) \,.
 \ee
Integrating this with respect $\eta_{r_n}$ gives
$$ \sum_v \eta_{r_n}(v)\P_\o( \o \gamma \mid \tau_{v} < \tau_\o^+) = c_\o \P_\o( \o \gamma)  \g_\o(\gamma_\ell,\eta_{r_n}) \to 
c_\o \P_\o( \o \gamma)  h(\gamma_\ell) \, ,$$
as $n \to \infty$ by our hypothesis. This yields the desired result by  \eqref{eq:hprocessvsSRW}. 
\end{proof}

Let $h$ be a potential and set $T_M = \inf \{t > 0: h(X_t) \geq M\}$ where 
$\{X_t\}$ is the network random walk.  Optional stopping for the nonnegative martingale $h(X_{t\wedge \tau_\o})$ gives $\P_v(T_M < \tau_\o)\leq h(v)/M$. Careful path reversal and conditioning yields the following.

\begin{lemma}\label{lem:dipolePotential}
Let $h$ be a potential for a  recurrent rooted network $(G,c,\o)$ and let $v\neq \o$ be a vertex. Then the network random walk $\{X_t\}$ satisfies 
$$ \P_\o( h(X_\ell) \geq M \mid \tau_v < \tau_\o^+) \leq {h(v) \over M} \, ,$$
for all $M>0$ and integer $0<\ell< d(\o,v)$, where $\tau_\o^+=\min \{ t \geq 1 : X_t = \o \}$.
\end{lemma}
\begin{proof} We may assume that $h(v)< M$. Since $\ell< d(\o,v)$ we have
\be\label{eq:beforeReversing}  \P_\o(h(X_\ell) \geq M \mid \tau_{v} < \tau_\o^+) \leq  \P_\o(T_{M} < \tau_{v} \mid \tau_{v} < \tau_\o^+) \, .\ee
Recall that $c_\o\P_\o(\tau_{v} < \tau_\o^+) = {c_{v}}\P_{v}(\tau_\o < \tau_{v}^+)$ by path reversal. Similarly, 
\begin{eqnarray}\label{eq:onemore} c_\o \P_\o(T_{M} < \tau_{v} < \tau_\o^+) &=& c_v \P_{v}(T_{M} < \tau_{\o} < \tau_{v}^+) \, , \qquad \text{so} \nonumber \\
\P_\o(T_{M} < \tau_{v} \mid \tau_{v} < \tau_\o^+) &=& \P_{v}(T_{M} < \tau_{\o} \mid \tau_{\o} < \tau_{v}^+) \, .\end{eqnarray}
Consider the random walk started at $v$ and write $L = \max\{t < \tau_\o : X_t = v\}$. Then $X_{L},\ldots, X_{\tau_\o}$ is distributed as a random walk started at $v$, conditioned on $\tau_\o < \tau_{v}^+$, and stopped at $\o$. Hence
\be\label{eq:DPstep2} \P_{v}(T_M < \tau_{\o} \mid \tau_{\o} < \tau_{v}^+) = \P_{v}(\exists s\in (L,\tau_\o) : h(X_s)\geq M ) \leq \P_{v}(T_M < \tau_o) \, .\ee
To bound the last probability we note that $S_t=h(X_{t \wedge \tau_o})$ is a non-negative martingale, so by optional stopping
$$ h(v) =  \E_v S_0 \geq \E_v S_{T_M} \geq M \P_{v}(T_M < \tau_\o) \, .$$ 
Combining this with \eqref{eq:beforeReversing}, \eqref{eq:onemore} and \eqref{eq:DPstep2} concludes the proof. \end{proof}

\subsection{Proof of \cref{lem:key}}
Recall that $\pot$ denotes the set of potentials on $(G,c,\o)$. Given an integer $r>0$, let $\Upsilon_r$ be the compact space of probability measures on $\partial B(\o,r)$. Observe that for every  $\psi:V\to \R$
\be\label{eq:trivial} \min_{v \in \partial B(\o,r)} \psi(v) =\min_{\eta \in \Upsilon_r} \sum_v \eta(v) \psi(v) \, .\ee
Next, we will apply  Sion's minimax theorem \cite{Sion58} to the bilinear continuous function $f_r:\pot\times \Upsilon_r\to \RR$, defined by 
$$ f_r(\psi, \eta) = \sum _{v \in \partial B(\o,r)} \eta(v)\psi(v) \, .$$
We obtain, using \eqref{eq:trivial}, that $M(r)$ defined in \eqref{eq:keylem} satisfies
\be\label{eq:minmax}M(r) = \max_{\psi\in \pot} \min_{\eta \in \Upsilon_r} f_r(\psi, \eta)=  \min_{\eta \in \Upsilon_r}\max_{\psi\in \pot}f_r(\psi, \eta) \, .\ee
For each integer $r>0$, let $\eta_r \in \Upsilon_r$ denote a minimizer of the right-hand side of \eqref{eq:minmax}. By \cref{claim:GreenIsDipole}, there exists a sequence $r_n\to \infty$ such that $\g_\o(\cdot,\eta_{r_n})$ converges pointwise to a function $h:V\to[0,\infty)$. Clearly $h$ is a potential. 

Fix $m>0$. By \cref{lem:hProcessInfty}, there exists an integer $\ell$ such that 
$ \P_{\mu_h}^h(h(Y_\ell) \geq 2m) \geq 1/2 \, .$ 
We now apply \cref{lem:PathsConverge} and sum \eqref{eq:pathsConverge} over all paths $\gamma=(\gamma_1,\ldots,\gamma_\ell)\subset S_h$ such that $h(\gamma_\ell)\geq 2m$. We infer that  
$$ \lim_{n\to \infty} \sum_v \eta_{r_n}(v) \P_\o(h(X_\ell) \geq 2m \mid \tau_v 
< \tau_\o^+)=\P_{\mu_h}^h(h(Y_\ell) \geq 2m) \geq  1/2 \, .$$
\cref{lem:dipolePotential} states that if $\ell < r_n$ and $v\in \partial B(\o,r_n)$, then
$$ \P_\o(h(X_\ell) \geq 2m \mid \tau_v 
< \tau_\o^+) \leq {h(v) \over 2m} \, .$$
Combining the last two inequalities yields that
$$ \liminf_{n \to \infty} f_{r_n}(h, \eta_{r_n}) = \liminf_{n \to \infty} \sum_v \eta_{r_n}(v) {h(v)}  \geq m \, .$$
 In conjunction with \eqref{eq:minmax}, this gives that $\liminf_n M(r_n) \geq  m$  for every $m>0$. Since     $M(r)$ is non-decreasing (\cref{remark:MR}), it must tend to infinity as $r \to \infty$. \qed

\section*{Acknowledgements} This work was supported by ERC consolidator grant 101001124 (UniversalMap) as well as ISF grants 1294/19 and 898/23. We thank Omer Angel, Ori Gurel-Gurevich and Fedja Nazarov for helpful discussions. 

\printbibliography

\end{document}